\documentclass[11pt,a4paper]{amsart}

\usepackage{amsfonts,amsmath,amssymb,amsthm}

\usepackage[utf8]{inputenc}
\usepackage[T1]{fontenc}

\usepackage{tikz}

\newcommand{\Cc}{\mathbb{C}} 

\newcommand{\Rr}{\mathbb{R}}

\newcommand{\Qq}{\mathbb{Q}}

\renewcommand {\le}{\leqslant}
\renewcommand {\ge}{\geqslant}

\newcommand{\Kk}{\mathbb{K}}

\newcommand{\defi}[1]{\emph{#1}}
\newcommand{\dist}{\operatorname{dist}}

{\theoremstyle{plain}
\newtheorem{theorem}{Theorem}            
\newtheorem*{theorem1bis}{Theorem 1'}            
\newtheorem{lemma}[theorem]{Lemma}       
}
{\theoremstyle{remark}

\newtheorem*{remark*}{Remark}  

}

\newcommand{\myfigure}[2]{
	\begin{center}
		\small
		\tikzstyle{every picture}=[scale=1.0*#1]
		#2
\end{center}}

\usepackage[a4paper]{geometry}
\title{Bilipschitz equivalence of polynomials}

\author{Arnaud Bodin}

\email{arnaud.bodin@univ-lille.fr}

\address{Universit\'e de Lille, CNRS, Laboratoire Paul Painlev\'e, 59000 Lille, France}

\subjclass[2010] {Primary 58K60 ; Sec. 12E05}

\keywords{Bilipschitz geometry, Polynomials, Moduli.}

\thanks{{\it Acknowledgment}. 
This work was supported by the ANR project ``LISA'' (ANR-17-CE40–0023-01).
}

\date{\today}

\begin{document}

\begin{abstract}
We study a family of two variables polynomials having moduli up to bilipschitz equivalence: two distinct polynomials of this family are not bilipschitz equivalent. However any level curve of the first polynomial is bilipschitz equivalent to a level curve of the second.
\end{abstract}

\maketitle


\section{Global bilipschitz equivalence}

Let $\Kk$ be $\Rr$ or $\Cc$.
For polynomial maps $f,g : \Kk^n \to \Kk$ we introduce two notions of bilipschitz equivalence: a level equivalence (a hypersurface $(f=c)$ is sent to a hypersurface $(g=c')$) and a global equivalence (any level $(f=c)$ is sent to another level $(g=c')$).

\begin{itemize}
  \item $\Kk^n$ is endowed with the Euclidean canonical metric. 	
	
  \item A map $\Phi : \Kk^n \to \Kk^n$ is \defi{Lipschitz} if there exists $K>0$ such that for all $x,y \in \Kk^n$:
  $$\| \Phi(x) - \Phi(y) \| \le K \|x-y\|.$$
  
  \item A map $\Phi : \Kk^n \to \Kk^n$ is \defi{bilipschitz} if it is a homeomorphism, Lipschitz and $\Phi^{-1}$ is also Lipschitz. Equivalently, $\Phi$ is bijective and  there exists $K>0$ such that
  $\frac1K \|x-y\| \le \| \Phi(x) - \Phi(y) \| \le K \|x-y\|$.
  
  \item Two sets $\mathcal{C}$ and $\mathcal{C}'$ of $\Kk^2$ are \defi{bilipschitz equivalent} if there exists a bilipschitz map $\Phi : \Kk^n \to \Kk^n$ such that  $\Phi(\mathcal{C}) = \mathcal{C}'$.
  
  \item Two functions $f,g : \Kk^n \to \Kk$ are \defi{right-bilipschitz equivalent} if there exists a bilipschitz map $\Phi : \Kk^n \to \Kk^n$ such that $g \circ \Phi = f$.
  
  \item Two functions $f,g : \Kk^n \to \Kk$ are \defi{left-right-bilipschitz equivalent} if there exist a bilipschitz map $\Phi : \Kk^n \to \Kk^n$ and a bilipschitz map $\Psi : \Kk \to \Kk$   such that $g \circ \Phi = \Psi \circ f$.  
\end{itemize}

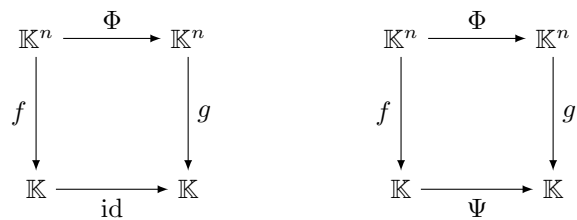
\begin{figure}[h]
\myfigure{0.4}{\begin{tikzpicture}[scale=1]

\begin{scope}
\node (A) at (0,0)  {$\Kk^n$};
\node (B) at (5,0)  {$\Kk^n$};
\node (AA) at (0,-5)  {$\Kk$};
\node (BB) at (5,-5)  {$\Kk$};
\draw[->,>=latex](A) -- (B) node[midway,above] {$\Phi$};
\draw[->,>=latex](AA) -- (BB) node[midway,below] {$\mathrm{id}$};
\draw[->,>=latex](A) -- (AA) node[midway,left]{$f$};
\draw[->,>=latex](B) -- (BB) node[midway,right]{$g$};
\end{scope}

\begin{scope}[xshift=12cm]
\node (A) at (0,0)  {$\Kk^n$};
\node (B) at (5,0)  {$\Kk^n$};
\node (AA) at (0,-5)  {$\Kk$};
\node (BB) at (5,-5)  {$\Kk$};
\draw[->,>=latex](A) -- (B) node[midway,above] {$\Phi$};
\draw[->,>=latex](AA) -- (BB) node[midway,below] {$\Psi$};
\draw[->,>=latex](A) -- (AA) node[midway,left]{$f$};
\draw[->,>=latex](B) -- (BB) node[midway,right]{$g$};
\end{scope}

\end{tikzpicture}}		

\caption{Two commutative diagrams. On the left: right-bilipschitz equivalence. On the right: left-right-bilipschitz equivalence.}
\end{figure}

Remarks:
\begin{itemize}
  \item A map $\Phi : \Kk^n \to \Kk^n$ can be $C^1$ but not Lipschitz. Hence (bi-)Lipschitz is \emph{not} an intermediate case between smooth and continuous. This is due to the non-compactness: for instance $\Phi : \Rr \to \Rr$, $x \mapsto x^2$ is $C^1$ but not Lipschitz.
  
  \item For similar reasons an algebraic automorphism of $\Kk^n$ does not necessarily provide a bilipschitz equivalence. For instance $f(x,y) = y$ and $g=y+x^2$ are algebraically equivalent using the map $\Phi : (x,y) \mapsto (x,y-x^2)$, but $\Phi$ is not bilipschitz.
\end{itemize}

It is clear that bilipschitz equivalence implies topological equivalence (i.e.\ when $\Phi$ and $\Psi$ are only homeomorphisms). The main question is: does topological equivalence implies bilipschitz equivalence? The answer is negative.

We will actually prove more. A theorem of Fukuda asserts that in a family of polynomials there is only a finite number of different types, up to topological equivalence, see \cite{Fu}, \cite{BT}. However the following theorem proves that the family of polynomials $f_s(x,y) = x(x^2y^2 -sxy-1)$ has moduli for bilipschitz equivalence, i.e.\ any two polynomials in this family are not right-bilipschitz equivalent. 

\begin{theorem}
	\label{th:main}
	Consider the family of polynomial in $\Kk[x,y]$:
	$$f_s(x,y) = x(x^2y^2 -sxy-1).$$
	\begin{itemize}
		\item $\Kk=\Rr$. 
		Any two polynomials $f_s$ and $f_{s'}$ with $s,s' \in \Rr$, $s\neq s'$ are not right-bilipschitz equivalent.
		However the special levels $(f_0=0)$ and $(f_1=0)$ are bilipschitz equivalent
		and the generic levels $(f_0=1)$ and $(f_1=1)$ are bilipschitz equivalent.
		
		\item $\Kk=\Cc$. 
		Fix $s\in\Cc$, with $s^2+3\neq0$. For all but finitely many $s'\in \Cc$, $f_s$ and $f_{s'}$ are not right-bilipschitz equivalent. However, if $s^2+4\neq0$ and $s'^2+4\neq0$, the polynomials $f_s$ and $f_{s'}$ are topologically equivalent.
	\end{itemize}
\end{theorem}

This is a version at infinity of a result by Henry and Parusi\'nski, \cite{HP}. Our polynomials $f_s$ have only one special level $(f_s=0)$ which plays the role of the singular level of the local examples of \cite{HP}. We recall that for a polynomial map $f :\Kk^n \to \Kk$ there is a notion of \defi{generic levels} $(f=c)$ and a finite number of \defi{special levels} whose topology is not the generic one. Special levels can be due to the presence of a singular point or to singularity at infinity as this the case in our examples. We will in fact prove a non bilipschitz equivalence ``at infinity'', after defining that two functions are bilipschitz equivalent at infinity if they are bilipschitz equivalent outside some compact sets.

\bigskip
\emph{Acknowledgments.} I thank Vincent Grandjean, Anne Pichon and Patrick Popescu-Pampu for their encouragements and the referees for their comments.

\section{Levels are bilipschitz equivalent}

Lemmas \ref{lem:level} and \ref{lem:levelgen} in this section will prove the bilipschitz real equivalence of theorem \ref{th:main}.
Let 
$$f_s(x,y) = x(x^2y^2 -sxy-1)$$
which, in this section, is considered as a family of polynomials in $\Rr[x,y]$.
\begin{lemma}
\label{lem:level}
The levels $(f_0=0)$ and $(f_1=0)$ are bilipschitz equivalent, that is to say	
there exists a bilipschitz map $\Phi : \Rr^2 \to \Rr^2$ such that
$\Phi((f_0=0)) = (f_1=0)$.
\end{lemma}
In other words, the (unique) special fibers of $f_0$ and $f_1$ are bilipschitz equivalent.

\begin{proof}
~ 

\textbf{Definition of $\Phi$.}

\begin{itemize}
	\item Let $\sigma = \frac{\sqrt5+1}{2}$ be the positive root of $z^2-z-1=0$.
	Let $\tau = \frac{\sqrt5-1}{2}$ be the positive root of $z^2+z-1=0$.

	\item We define a map $\Phi : \Rr^2 \to \Rr^2$ by the following formulas:
	\begin{itemize}
		\item For $(x,y) \in (xy=1)$ we define:
		$$\Phi(x,y) = (ax,by) \quad  \text{ with } ab = \sigma,$$
		such that $(a,b)$ depends on $(x,y)$ in the following way:
		$$
		\left\{
		\begin{array}{rl}
		(a,b) = (\sigma,1) & \text{ if } |x| \le \frac12 \\
		(a,b) = (1,\sigma) & \text{ if } |x| \ge 2 \\
		\end{array}
		\right.$$
		and extended to a smooth map for $\frac12  \le |x| \le 2$ so that the relation
		$ab=\sigma$ is always satisfied on $(xy=1)$.

		\item For $(x,y) \in (xy=-1)$ we similarly define $\Phi(x,y) = (ax,by)$ with $ab = \tau$, and $(a,b) = (\tau,1)$ for  $|x| \le \frac12$, $(a,b) = (1,\tau)$ for  $|x| \ge 2$ and extended in a smooth map for $\frac12  \le |x| \le 2$.
		
		\item $\Phi(0,y)=(0,y)$ for all $y\in \Rr$.
		
		\item $\Phi(x,y)=(x,y)$ for $(x,y)$ outside a neighborhood $\mathcal{N}$ of radius $1$ of
		$(xy=1) \cup (xy=-1)$.
		
		\item $\Phi$ is extended on $\mathcal{N}$ to a bilipschitz homeomorphism $\Phi : \Rr^2 \to \Rr^2$.  
	\end{itemize}

	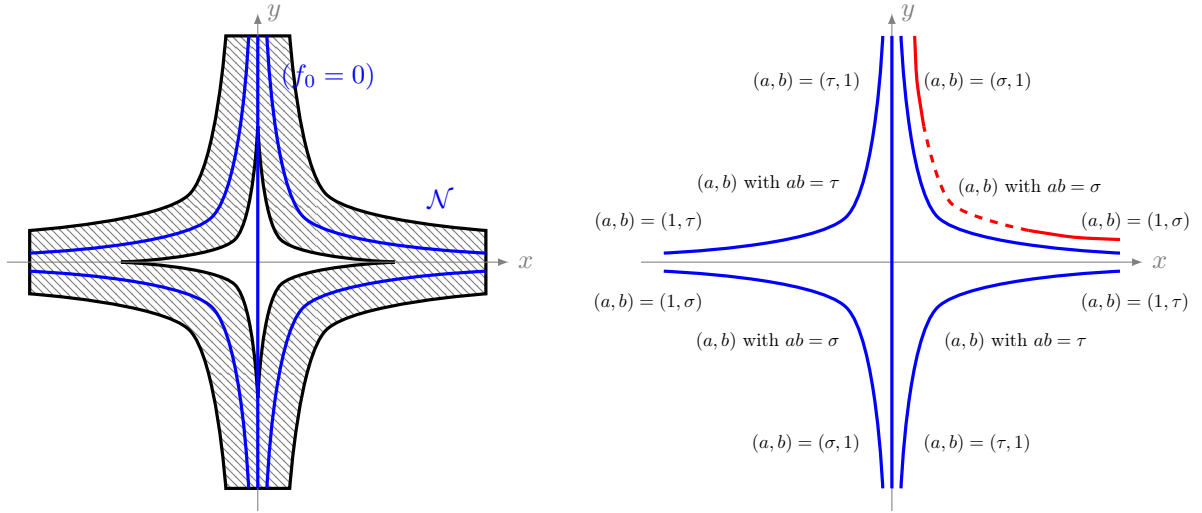
\begin{figure}[h]
	\begin{minipage}{0.45\textwidth}
		\myfigure{0.6}{\begin{tikzpicture}[scale=1]
\usetikzlibrary{patterns}

\def\bigstar{
plot [smooth] coordinates {(5,0.2+0.5)  (1+0.5,1+0.5)   (0.2+0.5,5) } -- 
plot [smooth] coordinates {(-0.2-0.5,5)  (-1-0.5,1+0.5)   (-5,0.2+0.5) } -- 
plot [smooth] coordinates {(-5,-0.2-0.5)  (-1-0.5,-1-0.5)   (-0.2-0.5,-5) } --
plot [smooth] coordinates { (0.2+0.5,-5) (1+0.5,-1-0.5)  (5,-0.2-0.5)  } -- cycle
}
\def\smallstar{
plot [smooth] coordinates {(3,0)  (0.5,0.5)   (0,3) } -- 
plot [smooth] coordinates {(0,3)  (-0.5,0.5)   (-3,0) } -- 
plot [smooth] coordinates {(-3,0)  (-0.5,-0.5)   (0,-3) } -- 
plot [smooth] coordinates {(0,-3)  (0.5,-0.5)   (3,0) } -- cycle
}
\draw[fill=black!15, even odd rule,pattern=north west lines, pattern color=gray] \bigstar \smallstar; 
\draw[very thick, black]  \bigstar;
\draw[very thick, black]  \smallstar;

\draw[->,>=latex,gray] (0,-5.5)--(0,5.5) node[right]{$y$};
\draw[->,>=latex,gray] (-5.5,0)--(5.5,0) node[right]{$x$};

\draw[very thick, blue] (0,-5) -- (0,5);
\draw[very thick, blue] plot [smooth] coordinates {(5,0.2)  (1,1)   (0.2,5) };
\draw[very thick, blue] plot [smooth] coordinates {(-5,0.2)  (-1,1)   (-0.2,5) };
\draw[very thick, blue] plot [smooth] coordinates {(5,-0.2)  (1,-1)   (0.2,-5) };
\draw[very thick, blue] plot [smooth] coordinates {(-5,-0.2)  (-1,-1)   (-0.2,-5) };

\node[blue,right] at (0.3,4.1) {$(f_{0}=0)$};
\node[blue,above] at (4,1) {$\mathcal{N}$};

\end{tikzpicture}}	
	\end{minipage}
	\qquad	
	\begin{minipage}{0.45\textwidth}
		\myfigure{0.6}{\begin{tikzpicture}[scale=1]

\draw[->,>=latex,gray] (0,-5.5)--(0,5.5) node[right]{$y$};
\draw[->,>=latex,gray] (-5.5,0)--(5.5,0) node[right]{$x$};

\draw[very thick, blue] (0,-5) -- (0,5);
\draw[very thick, blue] plot [smooth] coordinates {(5,0.2)  (1,1)   (0.2,5) };
\draw[very thick, blue] plot [smooth] coordinates {(-5,0.2)  (-1,1)   (-0.2,5) };
\draw[very thick, blue] plot [smooth] coordinates {(5,-0.2)  (1,-1)   (0.2,-5) };
\draw[very thick, blue] plot [smooth] coordinates {(-5,-0.2)  (-1,-1)   (-0.2,-5) };


\draw[very thick, red] plot [smooth] coordinates {(5,0.5)  (4,0.55)   (3,0.7) };
\draw[very thick, red,dashed] plot [smooth] coordinates {(3,0.7)  (1.3,1.3)   (0.7,3) };
\draw[very thick, red] plot [smooth] coordinates {(0.5,5)  (0.55,4)   (0.7,3) };

\node[above right,scale=0.7] at (4,0.55) {$(a,b)=(1,\sigma)$};
\node[above right,scale=0.7] at (1.3,1.3) {$(a,b)$ with $ab=\sigma$};
\node[right,scale=0.7] at (0.55,4) {$(a,b)=(\sigma,1)$};

\node[above left,scale=0.7] at (-4,0.55) {$(a,b)=(1,\tau)$};
\node[above left,scale=0.7] at (-1,1.4) {$(a,b)$ with $ab=\tau$};
\node[left,scale=0.7] at (-0.55,4) {$(a,b)=(\tau,1)$};

\node[below left,scale=0.7] at (-4,-0.55) {$(a,b)=(1,\sigma)$};
\node[below left,scale=0.7] at (-1,-1.4) {$(a,b)$ with $ab=\sigma$};
\node[left,scale=0.7] at (-0.55,-4) {$(a,b)=(\sigma,1)$};

\node[below right,scale=0.7] at (4,-0.55) {$(a,b)=(1,\tau)$};
\node[below right,scale=0.7] at (1,-1.4) {$(a,b)$ with $ab=\tau$};
\node[right,scale=0.7] at (0.55,-4) {$(a,b)=(\tau,1)$};

\end{tikzpicture}}		
	\end{minipage}	
	
	\caption{The definition of $\Phi$. Left: the level, a neighborhood of the level. Right: the values $(a,b)$ for the definition of $\Phi(x,y) = (ax,by)$ on the level.}
	\label{fig0506}
	\end{figure}

	\item The only point to prove is that the formulas actually yield a bilipschitz map around the axis.
	For instance let $(x_1,y_1) \in (xy=1)$ with $x_1>2$, so that $\Phi(x_1,y_1) = (x_1,\sigma y_1)$ and
	$(x_2,y_2) \in (xy=-1)$ with $x_2>2$ and $\Phi(x_2,y_2) = (x_2,\tau y_2)$.
	Then
	\begin{align*}
	\|\Phi(x_1,y_1) - \Phi(x_2,y_2)\| 
	& = \| (x_1-x_2,\sigma y_1-\tau y_2)\| \\
	& \le \| (x_1-x_2,2\sigma (y_1-y_2)\| \\
	& \le 2\sigma\| (x_1-x_2,y_1-y_2)\|
	\end{align*}
	(using that $y_1 - y_2 = |y_1| + |y_2|$). A similar bound holds for $\Phi^{-1}$ on this branch. 
	
	Then $\Phi : \Rr^2 \to \Rr^2$ is a bilipschitz homeomorphism.
	
\end{itemize}

\bigskip

\textbf{Equivalence.}

\begin{itemize}
	\item Let $f(x,y) =f_0(x,y) = x(x^2y^2-1)$ and
	$g(x,y)=f_1(x,y)=x(x^2y^2-xy-1)$.

	\item By definition of $\Phi$, $\Phi(0,y)=(0,y)$ so that the component
	$(x=0) \subset (f=0)$ is sent by $\Phi$ to $(x=0) \subset (g=0)$.
	
	\item Let $(x,y) \in (xy=1) \subset (x^2y^2=1) \subset (f=0)$.
	For such $(x,y)$, $\Phi(x,y)=(ax,by)$ with $ab=\sigma$.
	
	\item Let $\tilde g(x,y) = x^2y^2-xy-1$:
	$$\tilde g \circ \Phi(x,y) 
	= \tilde g(ax,by)
	= a^2b^2x^2y^2-abxy-1
	= \sigma^2(xy)^2-\sigma xy -1.$$
	As $xy=1$ we get:
	$$\tilde g \circ \Phi(x,y) = \sigma^2-\sigma -1=0,$$
	by definition of $\sigma$.
	Then $\Phi(x,y) \subset (\tilde g=0) \subset (g=0)$.
	A similar reasoning holds for $(xy=-1)$.	
\end{itemize}

\end{proof}

We now prove that two generic fibers are also bilipschitz equivalent.
\begin{lemma}
	\label{lem:levelgen}
	The levels $(f_0=1)$ and $(f_1=1)$ are bilipschitz equivalent, that is to say	
	there exists a bilipschitz map $\Phi : \Rr^2 \to \Rr^2$ such that
	$\Phi((f_0=1)) = (f_1=1)$.
\end{lemma}

\begin{proof}
~
\begin{itemize}
	\item \textbf{Parameterization of $(f_0=1)$.}
	The curve $(f_0=1)$ has equation $x^3y^2-x-1=0$ and 
	a parameterization $(x,y)$ is given by 
	$$y_+ = \sqrt{\frac{1}{x^2}+\frac{1}{x^3}}\quad  \text{ or }\quad 
	y_- = -\sqrt{\frac{1}{x^2}+\frac{1}{x^3}} \quad  \text{ for } x \in ]-\infty,-1] \ \cup \  ]0,+\infty[.$$
	 
	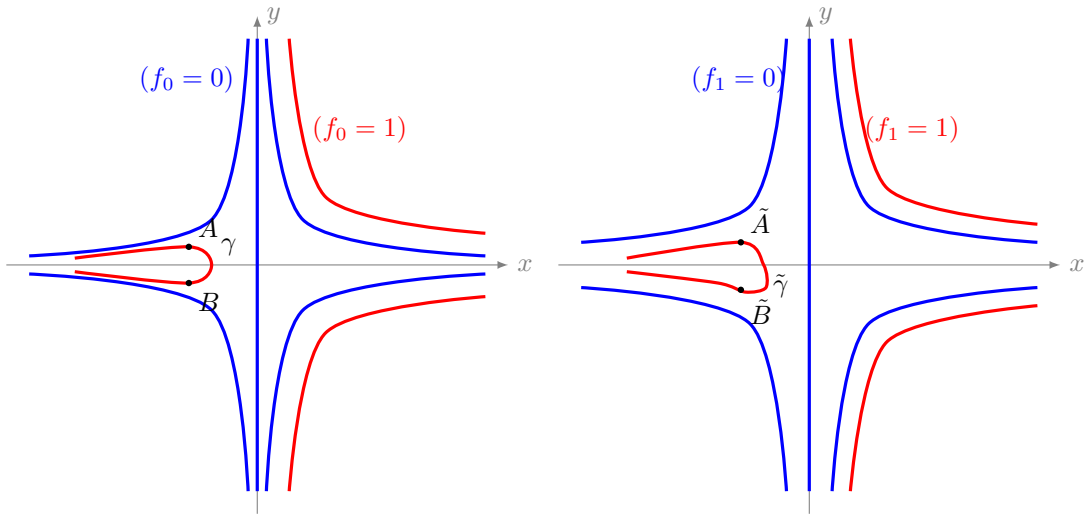
\begin{figure}[h]
	\begin{minipage}{0.45\textwidth}
		\myfigure{0.6}{\begin{tikzpicture}[scale=1]

\draw[->,>=latex,gray] (0,-5.5)--(0,5.5) node[right]{$y$};
\draw[->,>=latex,gray] (-5.5,0)--(5.5,0) node[right]{$x$};

\draw[very thick, blue] (0,-5) -- (0,5);
\draw[very thick, blue] plot [smooth] coordinates {(5,0.2)  (1,1)   (0.2,5) };
\draw[very thick, blue] plot [smooth] coordinates {(-5,0.2)  (-1,1)   (-0.2,5) };
\draw[very thick, blue] plot [smooth] coordinates {(5,-0.2)  (1,-1)   (0.2,-5) };
\draw[very thick, blue] plot [smooth] coordinates {(-5,-0.2)  (-1,-1)   (-0.2,-5) };

\draw[very thick, red] plot [smooth] coordinates {(5,0.2+0.5)  (1+0.5,1+0.5)   (0.2+0.5,5) };
\draw[very thick, red] plot [smooth] coordinates {(5,-0.2-0.5)  (1+0.5,-1-0.5)   (0.2+0.5,-5) };

\draw[very thick, red] plot [smooth] coordinates {(-4,0.15)  (-1.5,0.4)   (-1,0)  (-1.5,-0.4) (-4,-0.15) };

\node[blue,left] at (-0.3,4.1) {$(f_{0}=0)$};
\node[red,right] at (1,3) {$(f_{0}=1)$};

\fill (-1.5,0.4) circle (2pt) node[above right] {$A$};
\fill (-1.5,-0.4) circle (2pt) node[below right] {$B$};
\node[above right] at (-1,0) {$\gamma$};

\end{tikzpicture}}	
	\end{minipage}
	\quad	
	\begin{minipage}{0.45\textwidth}
		\myfigure{0.6}{\begin{tikzpicture}[scale=1]

\draw[->,>=latex,gray] (0,-5.5)--(0,5.5) node[right]{$y$};
\draw[->,>=latex,gray] (-5.5,0)--(5.5,0) node[right]{$x$};

\draw[very thick, blue] (0,-5) -- (0,5);
\draw[very thick, blue] plot [smooth] coordinates {(5,0.2+0.3)  (1+0.3,1+0.3)   (0.2+0.3,5) };
\draw[very thick, blue] plot [smooth] coordinates {(-5,0.2+0.3)  (-1-0.3,1+0.3)   (-0.2-0.3,5) };
\draw[very thick, blue] plot [smooth] coordinates {(5,-0.2-0.3)  (1+0.3,-1-0.3)   (0.2+0.3,-5) };
\draw[very thick, blue] plot [smooth] coordinates {(-5,-0.2-0.3)  (-1-0.3,-1-0.3)   (-0.2-0.3,-5) };

\draw[very thick, red] plot [smooth] coordinates {(5,0.2+0.7)  (1+0.7,1+0.7)   (0.2+0.7,5) };
\draw[very thick, red] plot [smooth] coordinates {(5,-0.2-0.7)  (1+0.7,-1-0.7)   (0.2+0.7,-5) };

\draw[very thick, red] plot [smooth] coordinates {(-4,0.15)  (-1.5,0.5)   (-1,0)  (-0.95,-0.5) (-1.4,-0.6) (-2,-0.4) (-4,-0.15) };

\node[blue,left] at (-0.3,4.1) {$(f_{1}=0)$};
\node[red,right] at (1,3) {$(f_{1}=1)$};

\fill (-1.5,0.5) circle (2pt) node[above right] {$\tilde A$};
\fill (-1.5,-0.55) circle (2pt) node[below right] {$\tilde B$};
\node[below right] at (-1,0) {$\tilde \gamma$};

\end{tikzpicture}}		
	\end{minipage}	
	
	\caption{The levels $(f_0=1)$ and $(f_1=1)$.}
	\label{fig0708}
	\end{figure}

	 \item\textbf{Parameterization of $(f_1=1)$.}
	 The curve 	$(f_1=1)$ has equation $x^3y^2-x^2y-x-1=0$,
	 a parameterization is given by: 
	 $$Y_+ = \frac{1}{2x} + \frac12\sqrt{\frac{5}{x^2}+\frac{4}{x^3}}\quad  \text{ or }\quad 
	 Y_- = \frac{1}{2x} - \frac12\sqrt{\frac{5}{x^2}+\frac{4}{x^3}} \quad  \text{ for } x \in ]-\infty,-\tfrac54] \ \cup \  ]0,+\infty[.$$
	 
	 \item \textbf{Definition of $\Phi$.}
	 
	 \begin{itemize}
	 	\item Case $x>0$. $\Phi$ is defined on $(f_0=1)$ using the parameterization by the formula $\Phi(x,y)= (x,Y_+)$, for $(x,y) \in (f_0=1)$ with $x>0$ and $y>0$;  $\Phi(x,y)= (x,Y_-)$, for $(x,y) \in (f_0=1)$ with $x>0$ and $y<0$.	 
	 		
	 	\item Case $x \le -2$. $\Phi$ is defined by the same formulas $\Phi(x,y)= (x,Y_+)$ (for $y>0$) or $\Phi(x,y)= (x,Y_-)$ (for $y<0$).
	 	
	 	\item Case $-2 \le x \le -1$.
	 	(Note: we do not use the above formulas in the neighborhood of the point $(-1,0)$ because the map $y_+ \mapsto Y_+$ is not bilipschitz near this point.)
	 	Let $A,B$ be the two points of $(x=-2)\cap(f_0=1)$. Let
	 	$\tilde A,\tilde B$ be their images by $\Phi$ (i.e.\ $A,B$ belong  $(x=-2)\cap(f_1=1)$). Let $\gamma$ be the compact part of $(f_0=1)$ between $A$ and $B$ and $\tilde \gamma$ be the compact part of $(f_1=1)$ between $\tilde A$ and $\tilde B$.	 	
	 	We extend $\Phi$ in a bilipschitz way from $\gamma$ to $\tilde \gamma$.
	 	This is possible as $\gamma$ and $\tilde \gamma$ are two compact connected components of a smooth algebraic curve. 
	 	$\Phi$ is now defined everywhere on $(f_0=1)$.
	 	
	 	\item We extend $\Phi$ on $\Rr^2$ to a bilipschitz map $\Phi : \Rr^2 \to \Rr^2$. For instance we may suppose $\Phi$ is the identity outside a tubular neighborhood or radius $1$ of $(f_0=1)$.
	 \end{itemize}

 	\item \textbf{Bilipschitz on $(f_0=1)$.}
 	It remains to justify that $\Phi$ is actually a bilipschitz map from $(f_0=1)$ to $(f_1=1)$. 
 	
 	\begin{itemize}
 		\item Case $x>0$ and $x\to 0$. Hence $y \to \pm\infty$.
 		Then $y_+ \sim \frac{1}{x^{3/2}}$ and $Y_+ \sim \frac{1}{x^{3/2}} \sim y_+$ so that the map $\Phi(x,y_+) = (x,Y_+)$ is bilipschitz.
 		The same applies for $y_-$ and $Y_-$.
 		
 		\item Case $x\to+\infty$. Hence $y\to 0$.
 		Then $y_+ \sim \frac1x$ and $Y_+ \sim \frac{\sqrt{5}+1}{2} \cdot \frac1x \sim \sigma y_+$. Then, as in the proof of proposition \ref{lem:level},
 		$\Phi(x,y_+) = (x,Y_+)$ is bilipschitz.
 		The same applies for $y_-$ and $Y_- \sim \tau y_-$ with $\tau =  \frac{\sqrt{5}-1}{2}$.
 		
 		\item Case $x \to -\infty$. It is similar to the previous case:
 		 $Y_+ \sim \tau y_+$, $Y_- \sim \sigma y_-$.
 		
 	\end{itemize}
\end{itemize}

\end{proof}

\section{Moduli}

The following theorem proves that under bilipschitz equivalence at infinity a family of polynomials can have moduli. It is a version at infinity of the example of Henry and Parusi\'nski \cite{HP}.
Two functions $f,g : \Kk^n \to \Kk$ are \defi{right-bilipschitz equivalent at infinity} if there exist compact sets $C,C'$ and a bilipschitz map $\Phi : \Kk^n\setminus C \to \Kk^n\setminus C'$ such that $g \circ \Phi = f$.

Using this notion, we will prove the moduli affirmation of theorem \ref{th:main} with the following refinement.
\begin{theorem1bis}
\label{th:moduli}

$$f_s(x,y) = x(x^2y^2 -sxy-1) \in \Kk[x,y].$$

\begin{itemize}
	\item $\Kk=\Rr$. 
	Any two polynomials $f_s$ and $f_{s'}$ with $s,s' \in \Rr$, $s\neq s'$ are not right-bilipschitz equivalent at infinity (hence not globally right-bilipschitz equivalent).
	Moreover they are also not left-right-equivalent if we assume $\Phi$ analytic at infinity.
	
	\item $\Kk=\Cc$. 
	Fix $s\in\Cc$, with $s^2+3\neq0$. For all but finitely many (explicit) $s'\in \Cc$, $f_s$ and $f_{s'}$ are not right-bilipschitz equivalent at infinity (hence not globally right-bilipschitz equivalent).	
\end{itemize}
\end{theorem1bis}

\subsection{Preliminaries}

\label{ssec:preliminaries}

\begin{itemize}
	\item Let $f_s(x,y) = x(x^2y^2-sxy-1) = x^3y^2 -sx^2y-x$.
	
	\item Then $\partial_x f_s(x,y) = 3x^2y^2 - 2sxy-1$.
	
	\item The equation $3z^2-2sz-1=0$ has discriminant $\Delta = 4(s^2+3)$ and two solutions:
	$$\alpha_s = \frac{s+\sqrt{s^2+3}}{3} \quad \text{and} \quad  \beta_s = \frac{s-\sqrt{s^2+3}}{3}.$$
	
	\item The polar curve $\Gamma_s : (\partial_x f_s = 0)$, associated to the projection on the $y$-axis, has two components:
	$$(xy=\alpha_s) \quad \text{and} \quad (xy=\beta_s),$$
	parameterized by:
	$$\left(\alpha_s t,\frac1t\right) \quad \text{and} \quad \left(\beta_s t,\frac1t\right) \qquad t \in \Kk\setminus\{0\}.$$
	
	\item We compute the values of $f_s$ on the polar components. Near the point at infinity $(0:1:0)$, that is to say for $t\to 0$, we compute the values of $f_s$ on each branch of $\Gamma_s$:
	$$f_s \left(\alpha_s t,\frac1t\right) = \alpha_s(\alpha_s^2-s\alpha_s-1) t,$$
	and	$$f_s \left(\beta_s t,\frac1t\right) = \beta_s(\beta_s^2-s\beta_s-1) t.$$
	
	\item We compare theses values for two branches at a same $y$-value:
	$$\frac{f_s \left(\alpha_s t,\frac1t\right)}{f_s \left(\beta_s t,\frac1t\right)}
	= \frac{\alpha_s(\alpha_s^2- s \alpha_s-1)}{\beta_s(\beta_s^2- s \beta_s-1)}.$$
	
	\item Our arguments will only focus on a neighborhood of a the point $(0:1:0)$ at infinity. More precisely we will say that an analytic curve $(x(t),y(t))$ \defi{tends to the point at infinity} $(0:1:0)$ if $y(t) \to +\infty$ and $\frac{|x(t)|}{|y(t)|} \to 0$ as $t\to 0$.
\end{itemize}

\subsection{Proof in the real case}
\label{ssec:proofreal}

\begin{itemize}
	\item Fix $t>0$. Let $A, B, C, D, E$ be the following points having all $y$-coordinates equal to $\frac1t$:
	\begin{itemize}
		\item $A \in (f_s=0)$ with $x_A>0$,
		\item $B \in \Gamma_s:(\partial_x f_s=0)$ with $x_B>0$,
		\item $C = (0,\frac1t) \in (f_s=0)$,
		\item $D \in \Gamma_s:(\partial_x f_s=0)$ with $x_D<0$,	
		\item $E \in (f_s=0)$ with $x_E<0$.		
	\end{itemize}

\begin{figure}[h]
	\myfigure{1}{\begin{tikzpicture}[scale=1]

\draw[very thick, blue] (0,0) -- (0,5);
\draw[very thick, blue] plot [smooth] coordinates {(-2,0)  (-0.8,2.5)   (-0.4,5) };
\draw[very thick, blue] plot [smooth] coordinates {(2,0)  (0.8,2.5)   (0.4,5) };

\draw[very thick, black] (-0.8,2.5) -- (0.8,2.5);

\draw[very thick, green!70!black] plot [smooth] coordinates {(0.3,0)  (0.3,2)   (0.4,2.5)  (0.6,2) (1,0)};
\draw[very thick, green!70!black] plot [smooth] coordinates {(-0.3,0)  (-0.3,2)   (-0.4,2.5)  (-0.6,2) (-1,0)};

\draw[thick, red] plot [smooth] coordinates {(0.7,0)   (0.4,2.5)  (0.2,5)};
\draw[thick, red] plot [smooth] coordinates {(-0.7,0)   (-0.4,2.5)  (-0.2,5)};

\coordinate (A) at (0.8,2.5);
\node[above right] at (A) {$A$};
\node[black, fill,circle,scale=0.3] at (A) {};

\coordinate (B) at (0.4,2.5);
\node[above] at (B) {$B$};
\node[black, fill,circle,scale=0.3] at (B) {};

\coordinate (C) at (0,2.5);
\node[above] at (C) {$C$};
\node[black, fill,circle,scale=0.3] at (C) {};

\coordinate (D) at (-0.4,2.5);
\node[above] at (D) {$D$};
\node[black, fill,circle,scale=0.3] at (D) {};

\coordinate (E) at (-0.8,2.5);
\node[above left] at (E) {$E$};
\node[black, fill,circle,scale=0.3] at (E) {};

\node[blue,below] at (-2.5,0) {$(f_s=0)$};
\node[blue,green!70!black,below] at (0.8,0) {$(f_s=c)$};
\node[blue,green!70!black,below] at (-0.8,0) {$(f_s=c')$};
\node[red,above] at (0.2,5) {$\Gamma_s$};
\end{tikzpicture}}		
	\caption{The situation for $f_s$.}
	\label{fig01}
\end{figure}
	
	\item Let us fix $s,s' \in \Rr$. 
	By contradiction let us assume that there exists a bilipschitz homeomorphism $\Phi : \Rr^2 \to \Rr^2$ such that $f_{s'} \circ \Phi = f_s$. 
	Let $K$ be its bilipschitz constant.
	Let $\tilde A, \tilde B,\ldots$ be the image by $\Phi$ of $A,B,\ldots$
	Let $\gamma$ be the segment $[AB]$ and $\tilde \gamma = \Phi(\gamma)$.
	
	\item $\Phi$ sends $(f_s=0)$ to $(f_{s'}=0)$ and, as it is a homeomorphism, it should send the component $(x=0)$ of  $(f_s=0)$ to the component $(x=0)$ of $(f_{s'}=0)$. Hence
	$x_{\tilde C} = 0$.
	
	\item $A,B,C,D,E$ and $\gamma$ are all included in the disk of radius $rt$ centered at 
	$C$, where $r$ is a constant that depends only on the fixed value $s$. Hence by the bilipschitz map $\Phi$,
	$\tilde A,\tilde B,\tilde C,\tilde D,\tilde E$ and $\tilde \gamma$ are all included in a disk of radius $Krt$ centered at $\tilde C$. 
	
	\item There is an issue: the point $B$ is on the polar curve $\Gamma_s$ but $\tilde B$ has no reason to be on $\Gamma_{s'}$. We will replace $\tilde B$ by a point $B'$ satisfying this condition.
	
\begin{figure}[h]
	\begin{minipage}{0.35\textwidth}
	\myfigure{0.8}{\begin{tikzpicture}[scale=1]

\draw[very thick, blue] (0,0) -- (0,5);
\draw[very thick, blue] plot [smooth] coordinates {(-2,0)  (-0.8,2.5)   (-0.4,5) };
\draw[very thick, blue] plot [smooth] coordinates {(2,0)  (0.8,2.5)   (0.4,5) };

\draw[very thick, black]plot [smooth] coordinates {(-0.9,2.2)  (-0.4,3)   (0,2.1)  (0.27,2.3) (0.4,2.8) (0.7,2.9)} ;

\draw[very thick, green!70!black] plot [smooth] coordinates {(0.3,0)  (0.3,2)   (0.4,2.5)  (0.6,2) (1,0)};
\draw[very thick, green!70!black] plot [smooth] coordinates {(-0.3,0)  (-0.3,2)   (-0.4,3)  (-0.6,2) (-1,0)};

\node[black,below] at (0.25,3.5) {$\tilde \gamma$};
\node[blue,below] at (-2.5,0) {$(f_{s'}=0)$};
\node[blue,green!70!black,below] at (0.9,0) {$(f_{s'}=c)$};
\node[blue,green!70!black,below] at (-0.8,0) {$(f_{s'}=c')$};

\end{tikzpicture}}	
	\end{minipage}
	\qquad	
	\begin{minipage}{0.35\textwidth}
	\myfigure{0.8}{\begin{tikzpicture}[scale=2.5]

\begin{scope}
\clip(-0.2,1.5) rectangle (2,3.5);

\draw[very thick, blue] (0,0) -- (0,5);
\draw[very thick, blue] plot [smooth] coordinates {(3,0)  (0.8,2.5)   (0.4,5) };

\draw[very thick, black]plot [smooth] coordinates {(-0.9,2.3)  (-0.4,3)   (0,2.1)  (0.29,2.25) (0.4,2.8) (0.65,2.9)} ;

\draw[very thick, green!70!black] plot [smooth] coordinates {(0.3,0)  (0.3,2)   (0.4,2.5)  (0.6,2) (1,0)};
\draw[very thick, green!70!black] plot [smooth] coordinates {(-0.3,0)  (-0.3,2)   (-0.4,3)  (-0.6,2) (-1,0)};

\draw[thick, red] plot [smooth] coordinates {(0.7,0)   (0.4,2.5)  (0.2,5)};
\draw[thick, red] plot [smooth] coordinates {(-0.7,0)   (-0.4,2.5)  (-0.2,5)};

\coordinate (A) at (0.65,2.9);
\node[above right] at (A) {$\tilde A$};
\node[black, fill,circle,scale=0.3] at (A) {};

\coordinate (B) at (0.32,2.3);
\node[left] at (B) {$\tilde B$};
\node[black, fill,circle,scale=0.3] at (B) {};

\coordinate (BB) at (0.4,2.5);
\node[above right] at (BB) {$B'$};
\node[black, fill,circle,scale=0.3] at (BB) {};

\coordinate (C) at (0,2.1);
\node[below left] at (C) {$\tilde C$};
\node[black, fill,circle,scale=0.3] at (C) {};

\end{scope}

\node[blue,below] at (-0.2,1.5) {$(f_{s'}=0)$};
\node[blue,green!70!black,below] at (1,1.5) {$(f_{s'}=c)$};
\node[red,above] at (0.3,3.5) {$\Gamma_{s'}$};

\end{tikzpicture}}		
	\end{minipage}	
	
	\caption{The situation for $f_{s'}$.}
	\label{fig0202}
\end{figure}
	
	\item Let $c = f_s(B)$. Let $\tilde X_{c}$ be the part of $(f_{s'}=c)$ in the ball of radius $Krt$ centered at $\tilde C$. As $f_{s'} (\tilde B) =f_s(B) = c$, then $\tilde B \in \tilde X_c$ and $\tilde X_c$ is non empty. Moreover $\tilde X_{c}$ is contained between two components of $(f_{s'}=0)$: $(x=0)$
	and one branch of $(x^2y^2 - s'xy - 1 = 0)$. Moreover $\tilde X_c$ is strictly below $\tilde \gamma$ except at $\tilde B$ (because $(f=c)$ is below $\gamma = [AB]$ and intersects it only at $B$). 
	
	\item Let $B'$ be the point of $\tilde X_c$ such that $y_{B'}$ is maximal among points of $\tilde X_c$. Then the tangent at $B'$ is horizontal, that is to say $\partial_x f_{s'}(B')=0$, hence $B' \in \Gamma_{s'}$. 
	Remember also that $B' \in \tilde X_c$ so that $f_{s'}(B') = c$.
	
	\item Partial conclusion: we constructed a point $B' \in \Gamma_{s'} \cap (f_{s'} = c)$ such that $\| B' - \tilde C \| \le Krt$ (with $x_{B'}>0$).
	
	\item We carry on the same proof for the other side. Let $c'=f_s(D)$, we find a point $D' \in \Gamma_{s'} \cap (f_{s'} = c')$ 
	such that $\| D' - \tilde C \| \le Krt$ (with $x_{D'}<0$).
	
	\item Now both these points $B'$ and $D'$ are in the same disk of radius $Krt$ centered at $\tilde C$. In particular:
	$$y_{D'} - 2Krt \le y_{B'} \le y_{D'} + 2Krt.$$
 		
	\item Let $B' = (\alpha_{s'}t',\frac1{t'})$ be the coordinates of $B'$ on the first branch of $\Gamma_{s'}$ and $D' = (\beta_{s'}t'',\frac1{t''})$ be the coordinates of $D'$ on the second branch of $\Gamma_{s'}$.
	The former inequalities rewrite:
	$$\frac{1}{t''} - 2Krt \le \frac{1}{t'} \le \frac{1}{t''} + 2Krt.$$	
	We consider $t\to 0$, so that $t'\to0$, $t'' \to 0$ (a neighborhood of $(0:1:0)$ is send to a neighborhood of $(0:1:0)$).
	Hence $t'' = t' + O(tt't'')= t' + O(t t'^2)$.
	
	\item Now
	$$\frac{f_{s'} \left(\alpha_{s'} t',\frac1{t'}\right)}{f_{s'} \left(\beta_{s'} t'',\frac1{t''}\right)}	
	= \frac{\alpha_{s'}t'\left( \alpha_{s'}^2-s'\alpha_{s'}-1  \right)}
	       {\beta_{s'}t''\left( \beta_{s'}^2-s'\beta_{s'}-1  \right)}$$
	$$       
	= \frac{\alpha_{s'}t'\left( \alpha_{s'}^2-s'\alpha_{s'}-1  \right)}
	       {\beta_{s'}(t'+ O(t t'^2))\left(\beta_{s'}^2-s'\beta_{s'}-1  \right)}
	\longrightarrow \frac{\alpha_{s'}(\alpha_{s'}^2-s'\alpha_{s'}-1)}{\beta_{s'}(\beta_{s'}^2-s'\beta_{s'}-1)}$$
	as $t'\to 0$.

	\item On the other hand:
	$$\frac{f_{s'}(B')}{f_{s'}(D')} = \frac{c}{c'} = \frac{f_{s'}(\tilde B)}{f_{s'} (\tilde D)}
			= \frac{f_{s}(B)}{f_{s}(D)} =  \frac{\alpha_s(\alpha_s^2-s\alpha_s-1)}{\beta_s(\beta_s^2-s\beta_s-1)}.$$
		
	Finally:  
	$$\frac{\alpha_s(\alpha_s^2-s\alpha_s-1)}{\beta_s(\beta_s^2-s\beta_s-1)} =  \frac{\alpha_{s'}(\alpha_{s'}^2-s'\alpha_{s'}-1)}{\beta_{s'}(\beta_{s'}^2-s'\beta_{s'}-1)}.$$

	\item The map $s \mapsto \frac{\alpha_s(\alpha_s^2-s\alpha_s-1)}{\beta_s(\beta_s^2-s\beta_s-1)}
	= \frac{2(s^2+3)\alpha_s+s}{2(s^2+3)\beta_s+s}$ is strictly decreasing for $s\in \Rr$ so that $s=s'$.
	
	\item Conclusion: if $s,s' \in \Rr$, with $s\neq s'$, then there exists no bilipschitz homeomorphism sending $f_s$ to $f_{s'}$. Since our arguments only care about situation near $(0:1:0)$ $f_s$ and $f_{s'}$ are not right-bilipschitz equivalent at infinity.
		
\end{itemize}

\subsection{No left-right-equivalence}
\label{ssec:noleftright}

We now prove that for $s\neq s'$ $f_s$ and $f_{s'}$ are not left-right-equivalent, if we ask the homeomorphism $\Phi$ to be analytic near the point at infinity $(0:1:0)$.
By contradiction we suppose that there exist bilipschitz homeomorphisms $\Phi$ and $\Psi$ such that $f_{s'} \circ \Phi = \Psi \circ f_{s}$ and $\Phi$ is analytic near the point at infinity $(0:1:0)$.
We continue with the same notation as above,
but we cannot conclude as before because we no longer have $\frac{f_{s'}(B')}{f_{s'}(D')}$ equal to $\frac{f_{s}(B)}{f_{s}(D)}$.

\begin{itemize}
	\item Let $C=(0,\frac1t)$ and $\Phi(C)=\tilde C = (0,\frac{1}{\tilde t})$ ($t>0$).
	The map $\frac 1t \mapsto \frac{1}{\tilde t}$ is a bilipschitz homeomorphism.
	We will assume $\Phi(0,0)=(0,0)$ so that 
	$\frac1K \frac1t \le \frac{1}{\tilde t} \le K \frac1t$
	hence $\frac1Kt \le \tilde t \le K t$. Define $\chi(t) = \tilde t$, for $t>0$, and set $\chi(0)=0$. 
	In the following we will actually only need the relation 
	$\frac1Kt \le \chi(t) \le K t$, but in fact the map $t\mapsto \chi(t)$ is a bilipschitz homeomorphism (with the constant $K^3$).
	
	\item We assumed that the map $\Phi$ is analytic at infinity around $(0:1:0)$. It implies that
	the map $t\mapsto \chi(t)$ is analytic for $t>0$: 
	$\chi(t) = a_0t^{r_0}+a_1t^{r_1}+\cdots$
	The map $\chi$ being bilipschitz it implies $r_0=1$ so that $\chi(t) = a_0t+a_1t^{r_1}+\cdots$ with $r_1>1$.
	
	\item Notice that the relation $f_{s'} \circ \Phi = \Psi \circ f_{s}$ implies that the map	$\Psi$ is also an analytic map. 
	
	\item Recall that $B=(\alpha_st,\frac1t)$ and $f_s(B)=c=\alpha_s(\alpha_s^2-s\alpha_s-1)t$,
	$D=(\beta_st,\frac1t)$ and $f_s(D)=c'=\beta_s(\beta_s^2-s\beta_s-1)t$.
	$\Phi(B)=\tilde B$ and $f_{s'}(\tilde B) = \tilde c = \Psi(c)$,
	$\Phi(D)=\tilde D$ and $f_{s'}(\tilde D) = \tilde c' = \Psi(c')$.
	We found $B' = (\alpha_{s'}t',\frac1{t'})$ close to $\tilde B$ such that $f_{s'}(B')=f_{s'}(\tilde B) = \tilde c$. Hence $\tilde c = \alpha_{s'}(\alpha_{s'}^2-s'\alpha_{s'}-1)t'$.
	Similarly $D'= (\beta_{s'}t'',\frac1{t''})$ is close to $\tilde D$ and $f_{s'}(D')=f_{s'}(\tilde D) = \tilde c'$. Hence $\tilde c' = \beta_{s'}(\beta_{s'}^2-s'\beta_{s'}-1)t''$.
	
	$B'$ is close to $\tilde B$ actually means $\left| \frac1{t'}-\frac{1}{\tilde t}\right| \le Krt$,
	that implies $|t'  - \tilde t| \le Krt t'\tilde t$. That implies $t' = \chi(t) + O(t^3)$.
	Similarly $t''= \chi(t) + O(t^3)$.
	
	\item The map $\Psi$ is defined, for negative values, by $c \mapsto \tilde c$ that is to say
	$\alpha_s(\alpha_s^2-s\alpha_s-1)t \mapsto \alpha_{s'}(\alpha_{s'}^2-s'\alpha_{s'}-1)t'$.
	It implies that, for $u<0$, the map $\Psi$ is defined by
	$$\Psi : u \mapsto  \alpha_{s'}(\alpha_{s'}^2-s'\alpha_{s'}-1)\chi\left(\frac{u}{\alpha_s(\alpha_s^2-s\alpha_s-1)}\right)
	+O(u^3).$$
	Hence, as $\chi(t) = a_0t+o(t)$:
	$$\Psi : u \mapsto  \frac{\alpha_{s'}(\alpha_{s'}^2-s'\alpha_{s'}-1)}{\alpha_s(\alpha_s^2-s\alpha_s-1)}u
	+o(u).$$
	Similarly $\Psi(d)=\tilde d$ so that for $u>0$:
	$$\Psi : u \mapsto  \frac{\beta_{s'}(\beta_{s'}^2-s'\beta_{s'}-1)}{\beta_s(\beta_s^2-s\beta_s-1)}u
	+o(u).$$
	
	\item By analycity of $\Psi$, it implies that the coefficients of $u$ are equal, whence 
	$$\frac{\alpha_s(\alpha_s^2-s\alpha_s-1)}{\beta_s(\beta_s^2-s\beta_s-1)} = \frac{\alpha_{s'}(\alpha_{s'}^2-s'\alpha_{s'}-1)}{\beta_{s'}(\beta_{s'}^2-s'\beta_{s'}-1)},$$
	which is impossible for $s\neq s'$ as we have seen before in section \ref{ssec:proofreal}.
\end{itemize}
	
\subsection{No left-right-equivalence (again)}
	
It is not clear whether $f_s$ and $f_{s'}$ ($s\neq s'$) are or not left-right bilipschitz equivalent when no restriction is made on $\Phi$.
However we can complicate our example in order to exclude left-right  equivalence.

\begin{lemma}
Let 
$$f_s(x,y) = x(x^4y^4-3sx^2y^2+1)$$
be a family of polynomials in $\Rr[x,y]$.
Then for $s,s'>1$, with $s\neq s'$, the polynomials
$f_s$ and $f_{s'}$ are not left-right bilipschitz equivalent.
\end{lemma}

\begin{proof}
The proof is similar to the proof of section \ref{ssec:noleftright}.
\begin{itemize}
	\item The equation $5z^4-9sz^2+1=0$ has $4$ real solutions 
	$-\alpha_s<-\beta_s<\beta_s<\alpha_s$ corresponding to $4$ branches of the polar curve $(\partial_x f_s=0)$.
	
	\item We use the same method as before in section \ref{ssec:noleftright} with $B=(-\alpha_st,\frac1t)$, $f_s(B)=-\alpha_s(\alpha_s^4-3s\alpha_s^2+1)t=
	c_s t>0$ and $D=(\beta_st,\frac1t)$, $f_s(D)=\beta_s(\beta_s^4-3s\beta_s^2+1)t=d_s t>0$ (with $t>0$). 
	
	\item This times for $u>0$ we have two formulas for $\Psi$ :
	$$\Psi(u) = c_{s'}\chi\left(\frac{u}{c_s}\right)+O(u^3),$$
	and 
	$$\Psi(u) = d_{s'}\chi\left(\frac{u}{d_s}\right)+O(u^3).$$
	
	\item It implies that the bilipschitz map $\chi$ verifies 
	$$\chi\left( \frac{c_s}{d_s} v \right) = \frac{c_{s'}}{d_{s'}} \chi(v) + O(v^3)$$
	for all $v>0$ near $0$.
	
	\item Then by lemma \ref{lem:biliaff} below, it implies $p=\frac{c_s}{d_s}>1$ is equal to $q=\frac{c_{s'}}{d_{s'}}>1$ which is impossible if $s\neq s'$.

\end{itemize}
\end{proof}

\begin{lemma}
\label{lem:biliaff}
Let $\chi: \Rr \to \Rr$ be a bilipschitz map such that
$$\chi(pv) = q\chi(v)+O(v^3)$$
for some constant $p,q>1$ and all $v$ near $0$.
Then $p=q$. 
\end{lemma}

\begin{proof}
We have $\chi(v) = q\chi(v/p)+ v^3\eta(v)$, where $\eta(v)$ is a bounded function for $v$ near $0$.
By induction it yields
$\chi(v) = q^n\chi(v/p^n) + v^3 \sum_{k=0}^{n-1} \eta(v/p^k) (q/p^3)^k$.
Hence, except for the special case $p^3=q$ that would be treated in a similar way, we have:
\begin{equation}
\left| \chi(v) - q^n\chi\left(\frac{v}{p^n}\right) \right| \le Cv^3 \frac{1-(q/p^3)^n}{1-q/p^3}.
\label{eq:chi}
\end{equation}
 
Let $K>0$ be a bilipschitz constant for $\chi$.
As $\chi(0)=0$ we have $K^{-1} < \frac{|\chi(v)|}{|v|} < K$ for all $v\neq0$.
In particular $K^{-1} < p^n\frac{|\chi(v/p^n)|}{|v|} < K$.

\emph{Case $p>q$.} Then we have $q^n\chi(v/p^n) \to 0$ as $n\to +\infty$.
At the limit, when $n\to +\infty$, inequality (\ref{eq:chi}) gives $|\chi(v)| \le C' v^3$,
which contradicts that $\chi$ is bilipschitz.

\emph{Case $p<q$.} Inequality (\ref{eq:chi}) gives
$$\left| \frac{p^n}{q^n}\chi(v) - p^n\chi\left(\frac{v}{p^n}\right) \right| \le C'v^3 \left( \frac{p^n}{q^n}- \frac{1}{p^{2n}} \right)$$
Fix $v\neq0$. As $n\to +\infty$, the term $p^n\chi(\frac{v}{p^n})$ does not tend towards $0$, it contradicts that all the other terms $\frac{p^n}{q^n}\chi(v)$, $\frac{p^n}{q^n}$ and $\frac{1}{p^{2n}}$ tends towards $0$. 

Conclusion: $p=q$.

\end{proof}

We completed the proof of theorem \ref{th:main} in the real setting.

\section{Proof in the complex case}

The proof in the complex case at infinity is an adaptation of the local proof of Henry and Parusi\'nski \cite{HP}.

\subsection{Notations}

\begin{itemize}
	\item Let $g : \Cc^2 \to \Cc$ be a polynomial map and $p = (x,y)$ be a point near the point at infinity $(0:1:0)$, that is to say $|y|\gg1$ and $|x|\ll |y|$.
	\item Fix $p_0$, let $c=g(p_0)$. Denote $B(p_0,\rho)$ the open ball centered at $p_0$ of radius $\rho$ and $X(p_0,\rho) = (g=c) \cap B(p_0,\rho)$.
	\item Fix $K>0$ and denote $\dist_{p_0,\rho,K}(p,q)$ the inner distance of 
	$p$ and $q$ supposed to be in the same connected component of $X(p_0,K\rho)$.
	\item Let
	$$\phi(p_0,K,\rho) = \sup \frac{\dist_{p_0,\rho,K}(p,q)}{\|p-q\|}$$
	be the ratio between the inner and outer distances.
	\item Denote 
	$$\psi(p_0,K,\rho) = \sup_{\rho'\le\rho} \phi(p_0,K,\rho').$$
	\item Finally let
	$$Y(\rho,K,A) = \{p \mid \psi(p,K,\rho) \ge A\}.$$
	be the set of points $p$ where the curvature of the curve $(g=c)$ is large.
	\item Let $\Phi : \Cc^2 \to \Cc^2$ be a bilipschitz homeomorphism at infinity such that $\tilde g \circ \Phi = g$. Let $L$ be a bilipschitz constant of $\Phi$. 
	\item Once $\Phi$ is fixed, we add a tilde to denote an object in then target space, for instance
	$$\tilde Y(\rho,K,A) = \{\tilde p \mid \psi(\tilde p,K,\rho) \ge A\}$$
	is the set of points $\tilde p$ in the target space where the curvature of the curve $(\tilde g=\tilde c)$ is large.
\end{itemize}

We have the following lemma saying that points with large curvature are sent to points of large curvature by a bilipschitz map:
\begin{lemma}[\cite{HP}, Lemma 2.1]
\label{lem:lem21}	
For $K \ge L^2$:
$$\tilde Y(L^{-1}\rho,K,AL^2) \subset \Phi(Y(\rho,K,A))	\subset \tilde Y(L\rho,K,AL^{-2}).$$
\end{lemma}
And a variant:
\begin{lemma}[\cite{HP}, Lemma 2.2]
\label{lem:lem22}	
Let $\delta>0$ and 	
$$Y(\delta,K,M,A) = \{p \mid \psi(p,M\|p\|^{-1+\delta},K) \ge A\}.$$
If $K \ge L^2$ then:
$$\tilde Y(\delta,K,ML^{-\delta},AL^2)  \subset  \Phi(Y(\delta,K,M,A)) \subset \tilde Y(\delta,K,ML^{+\delta},AL^{-2}).$$
\end{lemma}

Remarks:
\begin{itemize}
	\item There are two distinct uses of the norm:
	\begin{itemize}
		\item $\|p-q\|$: distance between two ``near'' points: a ``small'' number.
		\item $\|p\|$: distance to the origin: a ``large'' number. We will use it for $\frac{1}{\|p\|}$ in order to get a ``small'' number. 
	\end{itemize} 

	\item If we denote $\tilde p = \Phi(p)$, then the bilipschitz property implies:
	$L^{-1} \|p\| \le \| \tilde p \| \le L \|p\|$
	for some bilipschitz constant $L$, hence also: 
	$$L^{-1}\|p\|^{-1} \le \| \tilde p \|^{-1} \le L \|p\|^{-1}.$$

	\item Notice that in our definition of $Y(\delta,K,M,A)$ of lemma \ref{lem:lem22} there is a term in $\|p\|^{-1+\delta}$ while in \cite{HP} the term is $\|p\|^{1+\delta}$.
	After this modification, the proof is the same as in \cite{HP}.
	
	\item We will restrict ourselves to a neighborhood of the point at infinity $(0:1:0)$, in particular we may suppose $|y| \gg |x|$ so that 
	morally $\|p\| = \|(x,y)\| \simeq |y|$ (this is an equality in the case $\|\cdot\|=\|\cdot\|_\infty$).
\end{itemize}

Fix $s\in\Cc$ and denote $f_s(x,y) = x(x^2y^2-sxy-1)$.
Let us denote $U = \{(x,y) \mid |\partial_x f_s| <|\partial_y f_s| \}$.

\begin{lemma}[compare to \cite{HP}, Lemma 3.2]
\label{lem:lem32}	
Let $(x(t),y(t)) \in U$ with $y(t)=\frac1t$. Then for $s^2+3\neq0$:
$$x(t) = \gamma t + O(t^3) \quad \text{ and } \quad f_s(x(t),y(t)) = \gamma (\gamma^2-s\gamma-1)t + O(t^3),$$
with
$$\gamma = \alpha_s \text{ or } \gamma = \beta_s \text{ a solution of } 3z^2-2sz-1=0.$$
\end{lemma}

In this section we now suppose $s^2+3\neq0$.

\begin{proof}
Let $u = xy$. On $U$ the inequality $|\partial_x f_s| <|\partial_y f_s|$
yields $|3u^2-2su-1| < |x|^2 |2u-s|$. 
In a neighborhood of the point at infinity $(0:1:0)$ we first prove that $|x(t)|$ is bounded as $t\to +\infty$. If this is not the case, then write
$x(t) = a_0 t^{r_0} + a_1t^{r_1}+\cdots$ with $r_i \in \Qq$, $r_i<r_{i+1}$ and here $r_0<0$. As $y(t)=1/t$, then $u(t) \sim a_0t^{r_ 0-1} \to +\infty$.
Then $|3u^2-2su-1| < |x|^2 |2u-s|$ implies $r_0 \le -1$ in contradiction with $\frac{x(t)}{y(t)} \to 0$.
Now, as $|x(t)|$ is bounded, inequality $|3u^2-2su-1| < |x|^2 |2u-s|$ implies that $|u(t)|$ is also bounded. 
Write again $x(t) = a_0 t^{r_0} + a_1t^{r_1}+\cdots$ and using that $u(t)$ is bounded gives $r_0 \ge 1$: 
$x(t) = a_0 t + a_1t^{r_1}+\cdots$ and $u(t) = a_0 + a_1t^{r_1-1}+\cdots$ ($a_0\in\Cc$).
We plug $u(t)$ in the inequality $|3u^2-2su-1| < |x|^2 |2u-s|$:
$$\left| 3(a_0+a_1t^{r_1-1}+\cdots)^2 -2s(a_0+a_1t^{r_1-1}+\cdots)-1 \right| = O(t^2).$$
It implies:                                
$$3a_0^2-2sa_0-1=0$$
and 
$$6a_0a_1t^{r_1-1}-2sa_1t^{r_1-1} = O(t^2).$$
We may suppose $a_1\neq0$ and we now prove $r_1\ge3$.
Otherwise $6a_0=2s$, that is to say $s=3a_0$, but $a_0$ is a solution of $3z^2-2sz-1=0$.
This is only possible if $s^2+3=0$. 
So that $x(t) = \gamma t + O(t^3)$ as required, where $\gamma$ is a solution of $3z^2-2sz-1=0$.
Then $f_s(x(t),y(t)) = \gamma (\gamma^2-s\gamma-1)t + O(t^3)$.
\end{proof}

\begin{lemma}[compare to \cite{HP}, Lemma 3.3]
\label{lem:lem33}	
Let $0<\delta<1$ and $C>0$.
On the set:
$$\{ p = (x,y) \mid \exists p_0 = (x_0,y_0) \in U, f_s(p)=f_s(p_0), |y-y_0| \le C |y_0|^{-1+\delta}	\},$$
if we denote $y(t) = \frac1t$, then
\begin{equation}
\label{eq:lem33-1}
x(t) = O(t)
\end{equation}
and 
\begin{equation}
\label{eq:lem33-2}
f_s(x(t),y(t)) = \gamma(\gamma^2-s\gamma-1)t + O(t^{2-\delta}).
\end{equation}
\end{lemma}

\begin{proof}
We denote $y(t) = \frac1t$ and $y(t_0)=\frac1{t_0}$.
As $|y-y_0| \le C |y_0|^{-1+\delta}$, we have $|\frac1t-\frac1{t_0}| \le C |t_0|^{1-\delta}$
hence $|t_0/t - 1| \le C|t_0|^{2-\delta}$ hence $t_0/t \to 1$, i.e $t \sim t_0$.
Then $|t_0/t - 1| \le C'|t|^{2-\delta}$ so that $t_0 = t + O(t^{2-\delta})$.

Now by hypothesis and by lemma \ref{lem:lem32},
$$f_s(x(t),y(t)) 
= f_s(x(t_0),y(t_0))
= \gamma (\gamma^2-s\gamma-1)t_0 + O(t_0^3)
= \gamma (\gamma^2-s\gamma-1)t + O(t^{2-\delta}).$$

So that
$$f_s(x(t),y(t)) = x(t) (x(t)^2y(t)^2-sx(t)y(t)-1) = \gamma (\gamma^2-s\gamma-1)t + O(t^{2-\delta}).$$

We start over the computations of lemma \ref{lem:lem32}. Set $x(t) = a_0 t^{r_0} + a_1t^{r_1}+\cdots$ and $y(t)=1/t$.
Then
\begin{equation}
\frac{x(t)}{t} \left(\frac{x(t)^2}{t^2}-s\frac{x(t)}{t}-1\right) = \gamma (\gamma^2-s\gamma-1) + O(t^{1-\delta})
\label{eq:asym}
\end{equation}
We cannot have $r_0>1$ since we would have $\frac{x(t)}{t} \to 0$ (as $t\to 0$) and the left-hand side of equation (\ref{eq:asym}) would also tends to $0$.
We cannot either have $r_0<1$, since we would have $\left|\frac{x(t)}{t}\right| \to +\infty$ and the left-hand side of equation (\ref{eq:asym}) would also tends to infinity. 
Then $r_0=1$ and $a_0(a_0^2-sa_0-1)=\gamma (\gamma^2-s\gamma-1)$, so that $x(t) = O(t)$.

\end{proof}

\begin{lemma}[compare to \cite{HP}, Corollary 3.4]
\label{lem:cor34}
Let $Y=Y(\delta,K,M,A) = \{p \mid \psi(p,M\|p\|^{-1+\delta},K) \ge A\}$
where $0<\delta<1$, $M>0$ and $A$, $K$ are sufficiently large constants. Then the formulas (\ref{eq:lem33-1}) and (\ref{eq:lem33-2}) holds 
for $(x(t),y(t)) \in Y$ with $y(t) = \frac1t$.	
\end{lemma}

\begin{proof}
The proof is the same as in \cite{HP}: for $p_0=(x_0,y_0) \in Y$ there exists $p=(x,y) \in U$ such that 
$$\|p-p_0\| \le KM \|p_0\|^{-1+\delta}.$$
As $\frac12 |y_0| \le \|p_0\| \le 2 |y_0|$ (since $x_0 \le y_0$), it implies $|y-y_0|\le C |y_0|^{-1+\delta}$ and lemma \ref{lem:lem33} applies.
\end{proof}

\begin{lemma}[compare to \cite{HP}, Proposition 3.5]
\label{lem:prop35}
Let $Y=Y(\delta,K,M,A)$, where $0<\delta<1$, $M>0$ and $A$, $K$ are sufficiently large constants. Suppose that $p_1$ and $p_2$ are in $Y$ and there exists a $0<\delta_1<1$ such that $\| p_1 - p_2 \| \le \| p_1\|^{-1+\delta_1}$. Then for 
$\max \{\delta,\delta_1\} < \delta_2 < 1 $ and in a sufficiently small neighborhood of the point at infinity $(0:1:0)$:
$$\left| \frac{f_s(p_1)}{f_s(p_2)} - a \right| \le \| p_1 \|^{-1+\delta_2},$$
with 
$$a \in \left\{ 1, \frac{\alpha_s(\alpha_s^2-s\alpha_s-1)}{\beta_s(\beta_s^2-s\beta_s-1)}, \frac{\beta_s(\beta_s^2-s\beta_s-1)}{\alpha_s(\alpha_s^2-s\alpha_s-1)} \right\}.$$
\end{lemma}

\begin{proof}
Let $p_1 = (x_1(t),y_1(t))$ and $p_2=(x_2(t'),y_2(t'))$ be two points in $Y$.
Then by lemma \ref{lem:cor34} $$f_s(x_1(t),y_1(t))
=\gamma(\gamma^2-s\gamma-1)t + O(t^{2-\delta}),$$
$$f_s(x_2(t'),y_2(t'))
=\gamma'(\gamma'^2-s\gamma'-1)t' + O(t'^{2-\delta}),$$
where $\gamma$ and $\gamma'$ are in $\{\alpha_s,\beta_s\}$.

Now as $\| p_1 - p_2 \| \le \| p_1\|^{-1+\delta_1}$ it implies
$|y_1-y_2| \le 2|y_1|^{-1+\delta_1}$,  as in the proof of lemma \ref{lem:lem33} we get $t'=t+O(t^{2-\delta_1})$.
Whence 
$$f_s(x_2(t'),y_2(t')) = \gamma'(\gamma'^2-s\gamma'-1)t + O(t^{2-\delta_1}) + O(t^{2-\delta}).$$
Then
$$\frac{f_s(p_1)}{f_s(p_2)} = \frac{\gamma(\gamma^2-s\gamma-1)}{\gamma'(\gamma'^2-s\gamma'-1)}+O(t^{1-\delta_1})+O(t^{1-\delta}).$$

Then for $\delta_2>\max\{\delta,\delta_1\}$ with $\delta_2<1$ and in neighborhood of the point at infinity $(0:1:0)$ we get:
$$\left| \frac{f_s(p_1)}{f_s(p_2)} - a \right| \le \frac12|t|^{1-\delta_2} \le \| p_1 \|^{-1+\delta_2},$$
where $a=\gamma/\gamma'$.

\end{proof}

\begin{lemma}[compare to \cite{HP}, Lemma 3.6]
\label{lem:lem36}
Let $K$ and $A$ sufficiently large and $0<\delta<1$. Fix $s$ with $s^2+3\neq0$.
Then $Y = Y(\delta,K,M,A)$ is nonempty and contains the polar curve $\Gamma_s$.
Moreover all the limits	of $f_s(p_1)/f_s(p_2)$ given in lemma \ref{lem:prop35} can be obtained by taking $p_1$ and $p_2$ along the branches of $\Gamma_s$ associated to the point at infinity $(0:1:0)$.
\end{lemma}	

\begin{proof}
Fix $\delta$ and $K$. 
Let $\pi_c : (f_s=c) \to \Cc$ be the projection $(x,y) \mapsto y$. It is a triple covering branched at the points $\Gamma_s \cap (f_s=c)$.
These points are of coordinates
$$(\alpha_{s}t,\frac1{t}) \quad \text{ and } \quad 
(\beta_{s}t',\frac1{t'}) \quad \text{ with } \quad 
f_s(\alpha_{s}t,\frac1{t})= f_s(\beta_{s}t',\frac1{t'}) = c.$$
As $f_s(\alpha_{s}t,\frac1{t})=\alpha_s t(\alpha_s^2-s\alpha_s-1)$ it implies 
$$t = \frac{c}{\alpha_s (\alpha_s^2-s\alpha_s-1)}
 \quad \text{ and similarly } \quad 
t' = \frac{c}{\beta_s (\beta_s^2-s\beta_s-1)}.$$
For $s^2+3\neq0$, $\alpha_s \neq \beta_s$ and it also implies $t \neq t'$  hence
$|y(t)-y(t')|$ is of order $y(t)$, that is to say two points of ramifications are far enough.
Let $p_0=(x_0,y_0)$ be a point of ramification of $\pi_c$. 
Let $\mathcal{V} = \{y \mid |y-y_0| \le \epsilon |y_0|\}$, with $\epsilon$ sufficiently small such that no other ramification point projects in $\mathcal{V}$.
For a sufficiently large $p_0$ (i.e.\ small $c$),
$X(p_0,KM\|p_0\|^{-1+\delta}) \subset \pi_c^{-1}(\mathcal{V})$. 
Now let $p=(x,y)$ such that:
$$|y-y_0| \le |y_0|^{-1+\delta},$$
then by lemma \ref{lem:lem33}, $x = O(\frac1y)$ whence
$$\|p-p_0\| \le 2\|p_0\|^{-1+\delta}.$$
Let $\mathcal{V}_\delta = \{y \mid |y-y_0| \le \epsilon |y_0|^{-1+\delta}\}$,
by the above inequality we get $\pi_c^{-1}(\mathcal{V}_\delta) \subset X(p_0,KM\|p_0\|^{-1+\delta})$.
We restrict the triple branched covering $\pi_c$ to a map $\tilde \pi_c$ from $\pi_c^{-1}(\mathcal{V}_\delta)$ composed by only two components of the triple cover.
Let $y \in \mathcal{V}_\delta$ such that $|y-y_0| = \frac12 |y_0|^{-1+\delta}$.
Let $p_1 = (x_1,y)$, $p_2=(x_2,y)$ be the two points of $\tilde \pi_c^{-1}(y)$.
These two points are in $\tilde \pi_c^{-1}(\mathcal{V}_\delta)$ which is a connected set.
Any curve $\gamma$ in $\tilde \pi_c^{-1}(\mathcal{V}_\delta)$ from $p_1$ to $p_2$ passes through $p_0$, hence the projection of $\gamma$ by $\tilde \pi_c$ passes through $y_0$.
Hence the inner distance (in $(f_s=c)$) of $p_1$ and $p_2$ is greater or equal than
$2|y-y_0|$, it yields:
$$\dist_{p_0,M\|p_0\|^{-1+\delta},K}(p_1,p_2) \ge 2|y-y_0| = |y_0|^{-1+\delta}= |t|^{1-\delta},$$
where we denote $y_0=\frac1t$.
By lemma \ref{lem:lem33} we have $x_1 = O(t)$ and $x_2 = O(t)$, so that
$$\|p_1-p_2\| \le C|t|.$$
Then
$$\frac{\dist_{p_0,M\|p_0\|^{-1+\delta},K}(p_1,p_2)}{\|p_1-p_2\|} \ge \frac{1}{C|t|^\delta} \xrightarrow[t \to 0]{} + \infty.$$
Then $\psi(p_0,M\|p_0\|^{-1+\delta},K) \to + \infty$, as $p_0$ tends to the point at infinity $(0:1:0)$. It means that the branch of $\Gamma_s$ near this point at infinity is included in $Y(\delta,K,M,A)$.

Finally we have already proved in subsection \ref{ssec:preliminaries} that the list of values $f_s(p_1)/f_s(p_2)$ on $\Gamma_s$ is the required one.
\end{proof}

We conclude by the proof of the theorem in the complex case.	
\begin{proof}[Proof of theorem \ref{th:main}']
Fix $s$. By lemma \ref{lem:lem22} the set $Y$ for $f_s$ is sent into a set	$\tilde Y$ for $f_{s'}$. The polar curve $\Gamma_s$ is included in $Y$ (lemma \ref{lem:lem36}) and on this polar curve
$f_s(p_1)/f_s(p_2)$ tends to a  $\frac{\alpha_s(\alpha_s^2-s\alpha_s-1)}{\beta_s(\beta_s^2-s\beta_s-1)}$ for instance (lemma \ref{lem:prop35}).
On the one hand $f_{s'}(\tilde p_1)/f_{s'}(\tilde p_2)$ tends to the same value, because the bilipschitz homeomorphism $\Phi$ sends the levels of $f_s$ to the levels of $f_{s'}$.
On the other hand $\tilde p_1$, $\tilde p_2$ are in $\tilde Y$ so that
$f_{s'}(\tilde p_1)/f_{s'}(\tilde p_2)$ is in 
$$\left\{ 1, \frac{\alpha_{s'}(\alpha_{s'}^2-s'\alpha_{s'}-1)}{\beta_{s'}(\beta_{s'}^2-s'\beta_{s'}-1)}, \frac{\beta_{s'}(\beta_{s'}^2-s'\beta_{s'}-1)}{\alpha_{s'}(\alpha_{s'}^2-s'\alpha_{s'}-1)} \right\}.$$
This is only possible for a finite set of values $s'$.
\end{proof}

\section{Topological equivalence}

To complete the complex part of theorem \ref{th:main} we prove the topological equivalence of any two polynomials.

\begin{lemma}
\label{th:topeq}
Consider the following family of polynomials in $\Cc[x,y]$: 
$$f_s(x,y) = x(x^2y^2 -sxy-1)$$
with $s^2+4\neq0$.
For any $s$ and $s'$ the polynomials $f_s$ and $f_{s'}$ are topologically equivalent.
\end{lemma}
This family is similar to examples in \cite{Bo1} of polynomials that are topologically equivalent but not algebraically equivalent.
Recall that two polynomials $f,g : \Kk^n \to \Kk$ are \defi{topologically equivalent} if there exist a homeomorphism $\Phi : \Kk^n \to \Kk^n$ and a homeomorphism $\Psi : \Kk \to \Kk$ such that $g \circ \Phi = \Psi \circ f$. 
We will use the following result that is global version of Lê-Ramanujam $\mu$-constant theorem. See \cite{Bo2} for the two variables case and \cite{BT} for any number of variables.
\begin{theorem}
\label{th:mucst}
Let $\{f_s\}_{s \in [0,1]}$ be a continuous family of complex polynomials with isolated singularities (in the affine space and at infinity), with $n \neq 3$ variables.
Suppose that the following integers are constant w.r.t.\ the value of $s$:
\begin{itemize}
	\item $\deg f_s$, the degree,
	\item $\# \mathcal{B}_s$, the number of irregular values,
	\item $\chi(f_s=c_{\text{gen}})$, the Euler characteristic of a generic fiber.
\end{itemize}
Then $f_0$ and $f_1$ are topologically equivalent.
\end{theorem}

\begin{proof}[Proof of lemma \ref{th:topeq}]
	~
\begin{itemize}
	
	\item \textbf{Degree.} It is clear that the degree of the $f_s$ is independent of $s$.
	
	\item \textbf{Affine singularities.}
	We search for points $(x,y)$ where both derivatives vanish.
	$\partial_x f_s (x,y) = 3x^2y^2-2sxy-1$ and $\partial_y f_s(x,y) = x^2(2xy-s)$.
	If $x=0$ then $\partial_x f_s(x,y) \neq 0$.
	So that $\partial_y f_s(x,y)=0$ implies $2xy-s=0$. We plug $xy=s/2$ in  $\partial_x f_s (x,y)=0$ and get $s^2+4=0$.
	Notice that $s^2+4=0$ gives also the values where $f_s$ is not a reduced polynomial.	
	Conclusion: for $s^2+4\neq0$, the polynomials $f_s$ has no affine singularities (nor affine critical values), so that its global affine Milnor number is $\mu_s=0$.
	
	\item \textbf{Singularities at infinity.}
	The two points at infinity for this family are $P_1 = (0:1:0)$ and $P_2=(1:0:0)$.
	Let $F_s(x,y,z) = x(x^2y^2-sxyz^2-z^4) - cz^5$ be the homogenization of $f_s(x,y)-c$. 
	
	\begin{itemize}
		\item \textbf{Milnor number at $P_1$.}
		We localize $F_s$ at $P_1=(0:1:0)$ to get
		$g_s(x,z)=F_s(x,1,z)=x(x^2-sxz^2-z^4)-cz^5$.
		We compute the local Milnor of $g_s$ at $(0,0)$.
		For instance we may use the Newton polygon of $g_s$ and Kouchnirenko formula.
		We get, for any $s$ (with $s^2+4\neq0$) and depending on $c$:
		$$\mu(g_s)=8 \  \text{ if } c \neq 0 \quad \text{ and } \quad  \mu(g_s)=10 \ \text{ if } c = 0.$$
		Hence the value $0$ is an irregular value at infinity and the jump of Milnor number is $\lambda_{P_1} = 10-8 = 2$.
		
	\begin{center}
	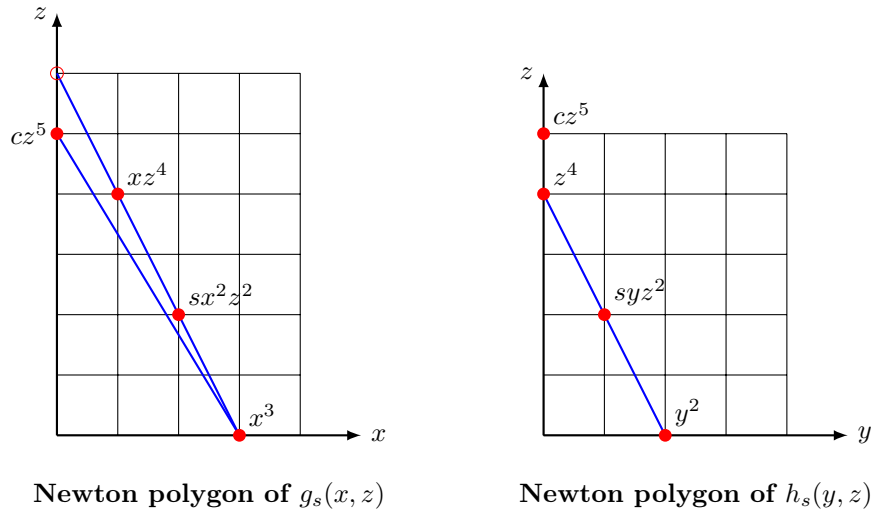
\begin{figure}[h]
	\myfigure{0.8}{\begin{tikzpicture}[scale=1]

\def\rayon{3 pt}

\begin{scope}
\draw (0,0) grid (4,6);
\draw[->,>=latex,thick] (0,0) -- (5,0) node[right] {$x$};
\draw[->,>=latex,thick] (0,0) -- (0,7) node[left] {$z$};

\draw[blue,thick] (0,5)--(3,0);
\draw[blue,thick] (0,6)--(3,0);

\fill[red] (3,0) circle (\rayon) node[black,above right] {$x^3$};
\fill[red] (2,2) circle (\rayon) node[black,above right] {$sx^2z^2$};
\fill[red] (1,4) circle (\rayon) node[black,above right] {$xz^4$}; 
\fill[red](0,5) circle (\rayon) node[black,left] {$cz^5$};
\draw[red](0,6) circle (\rayon); 

\node at (2.5,-1) {\bf Newton polygon of $g_s(x,z)$};
\end{scope}

\begin{scope}[xshift=8cm]
\draw (0,0) grid (4,5);
\draw[->,>=latex,thick] (0,0) -- (5,0) node[right] {$y$};
\draw[->,>=latex,thick] (0,0) -- (0,6) node[left] {$z$};

\draw[blue,thick] (0,4)--(2,0);

\fill[red] (2,0) circle (\rayon) node[black,above right] {$y^2$};
\fill[red] (1,2) circle (\rayon) node[black,above right] {$syz^2$};
\fill[red] (0,4) circle (\rayon) node[black,above right] {$z^4$}; 
\fill[red](0,5) circle (\rayon) node[black,above right] {$cz^5$}; 

\node at (2.5,-1) {\bf Newton polygon of $h_s(y,z)$};
\end{scope}

\end{tikzpicture}}	

	\caption{Computation of Milnor number at infinity.}
	\label{fig04}
	\end{figure}	
	\end{center}	

	   \item \textbf{Milnor number at $P_2$.}
	   At $P_2=(1:0:0)$ we get
	   $h_s(y,z)=F_s(1,y,z) = y^2-syz^2-z^4-cz^5$.
	   The local Milnor number of $h_s$ at $(0,0)$ is independent of $s$ and $c$: 
	   $$\mu(h_s)=3.$$
	   So that there is no irregular values at infinity for this point and $\lambda_{P_2} = 0$.
	   
	   \item Then the Milnor number at infinity is $\lambda_s = \lambda_{P_1}+\lambda_{P_2}=2$ and the only irregular value at infinity is $0$.
	\end{itemize}   
	  	
	\item \textbf{Conclusion.}
	For all $s$ the only irregular value is $0$: $\mathcal{B}_s = \{0\}$,
	the Euler characteristic of a generic fiber given by
	$\chi_s = 1 - \mu_s - \lambda_s = -1$ is also constant.
	Then by theorem \ref{th:mucst} any $f_s$ and $f_{s'}$ are topologically equivalent.

\end{itemize}	
	
\end{proof}


\bibliographystyle{plain}
\bibliography{globi.bib}

\end{document}